\documentclass{iopart}
\usepackage{graphicx}
\usepackage{amsthm, soul} % Sara da provare a eliminare

\usepackage{verbatim}
\usepackage{color}
\usepackage[colorinlistoftodos]{todonotes}
\usepackage{tikz}

\newtheorem{proposition}{Proposition}
\newtheorem{theorem}{Theorem}
\newtheorem{definition}{Definition}
\newtheorem{corollary}{Corollary}
\theoremstyle{definition} 
\newtheorem{remark}{Remark}

\newcommand{\mbf}[1]{\mathbf{#1}}

\definecolor{dgreen}{rgb}{0,0.20,0}

\begin{document}

\title{On the two-step estimation of the cross--power spectrum for dynamical inverse problems}
\author{Elisabetta Vallarino$^1$, Sara Sommariva$^1$, Michele Piana$^{1,2}$ and Alberto Sorrentino$^{1,2}$}
\address{$^1$ Dipartimento di Matematica, Universit\`a di Genova, Italy}
\address{$^2$ CNR--SPIN, Genova, Italy}
\date{November 2018}

\begin{abstract}
We consider the problem of reconstructing the cross--power spectrum of an unobservable multivariate stochatic process from indirect measurements of a second multivariate stochastic process, related to the first one through a linear operator. In the two--step approach, one would first compute a regularized reconstruction of the unobservable signal, and then compute an estimate of its cross--power spectrum from the regularized solution. We investigate whether the optimal regularization parameter for reconstruction of the signal also gives the best estimate of the cross--power spectrum. 
We show that the answer depends on the regularization method, and specifically we prove that, under a white Gaussian assumption: (i) when regularizing with truncated SVD the optimal parameter is the same; (ii) when regularizing with the Tikhonov method, the optimal parameter for the cross--power spectrum is lower than half the optimal parameter for the signal. We also provide evidence that a one--step approach would likely have better mathematical properties of the two--step approach. Our results apply particularly to the brain connectivity estimation from magneto/electro-encephalographic recordings and provide a formal interpretation of recent empirical results.
\end{abstract}
{\it Keywords\/}: regularization theory, multivariate stochastic processes, cross--power spectrum, magneto--/electro--encephalography (M/EEG), functional connectivity.

\section{Introduction}

Dynamical inverse problems are typically concerned with two interplaying and in some sense still open issues. The first one is related to the reconstruction of the unobserved, multivariate stochastic process from the measured time series; the second one is the estimate of the statistical interdependence of the individual components of the  multivariate stochastic process. A paradigmatic example of these issues is the estimate of brain functional connectivity from recordings of magneto/electro-encephalographic (M/EEG) data, currently a hot topic in neuroscience.  Functional connectivity is systematically used to study both the healthy \cite{de2010temporal} and the pathological \cite{st_10, vanMi_etal19} brain, either at rest \cite{br_11} or during the execution of specific tasks \cite{li_etal15, wahe_15}. 

While there is no unique formal definition of functional connectivity, the term is generally used to identify various forms of statistical interdependence between the temporal waveforms of spatially distinct brain areas \cite{sakkalis_11}.
In the last couple of decades, such interdependence is increasingly studied in the frequency domain; this makes sense in light of the increasingly accepted model that neural interactions between brain regions are mediated by synchronization of their rhythmic activity in specific frequency bands \cite{fr05}.
However, M/EEG only record the magnetic field/electric potential at the scalp; therefore functional connectivity between brain regions has to be estimated indirectly, using the scalp data and the physical model that relates neural currents to the recordings.

In this framework, the majority of connectivity studies employs a two--step approach \cite{sc_gr19}: first, an estimate of the source time courses is obtained using an inverse method; then, frequency--domain connectivity metrics are computed from the cross--spectrum of the reconstructed source time courses. Due to the multitude of available inverse methods  \cite{hail94, vvetal97, calvetti2015hierarchical, costa2017bayesian, bekhti2018hierarchical, luetal19, sorrentino2017inverse, hapu_17,  ilsa19} and connectivity metrics \cite{baccala2001partial, nolte04, chella2014third, geweke1982measurement, pereda_05, sakkalis_11}, in the last decade there has been growing interest in validating and comparing different combinations of methods \cite{fraschini2016effect, hietal17,ch_19, sommariva2019comparative, nuetal19}.

Recent empirical evidence suggests that the two--step approach might feature an unexpected parameter tuning issue. Indeed, the cure of ill--posedness requires a tradeoff between solution complexity and data fitting, and this tradeoff is realized by means of the selection of the optimal regularization parameter. 
It would seem natural that the optimal estimate of the cross--spectrum can only be attained with the optimal reconstruction of the signal. Yet, in a recent study \cite{hietal16} have shown that 
the value of the regularization parameter that provides the best reconstruction of the source spectral power does not coincide with the value that provides the best reconstruction of the source-level functional connectivity quantified through coherence.

Motivated by this empirical result, in this work we investigate the following problem: let $\mbf{Y}(t)$ be noisy and indirect measurements of a multivariate stochastic process $\mbf{X}(t)$; let $\mbf{x}_{\lambda}(t)$ be the reconstruction of the hidden signal, obtained by means of a regularization algorithm; finally, assume that the cross--spectrum of $\mbf{X}(t)$, denoted as $\mbf{S}^{\mbf{X}}(f)$, is estimated from the reconstructed signal $\mbf{x}_{\lambda}(t)$; under these conditions, does the optimal regularization parameter for reconstruction of the hidden signal coincide with the optimal regularization parameter for reconstruction of its cross--spectrum?

In particular, we will prove that the answer is ``no'' when the regularized solution is computed via Tikhonov regularization, thus confirming the empirical results of \cite{hietal16}. We will also prove that the answer is ``yes'' when the regularized solution is computed via truncated Singular Value Decomposition, thus showing that the answer to the question actually depends on the choice of the inverse method. In addition, we will show that a one-step approach relying on a mathematical model directly relating the measured data to the unknown cross-spectrum should be preferred. In particular, a preliminary analysis performed in this paper shows that the one-step approach enhances the filtering effectiveness of regularization with respect to the standard two-step approach.

The structure of the article is as follows: in Section 2 we provide the general definitions and formalize the main question of the paper. In Section 3 we express the reconstructions errors in terms of the filter factors and provide an interpretation. Section 4 contains the main results of our work: we show that the optimal regularization parameters for reconstruction of $\mbf{x}(t)$ and $\mbf{S}^{\mbf{x}}(f)?$ are generally different and that this difference depends on the inversion method. In Section 5 we show how the filter factors of the two--step approach have a jittering behaviour, while those of a possible one--step approach would be smooth. Our conclusions will be offered in Section 6, together with the discussion of possible directions for future work.

%%%%%%%%%%%%%%%%%%%%%%%%%%%%%%%%%%%%%%%%%%%%%%%%%%%%%%%%%%%%%%%%%%%%%%%%%%%%%%%%%%%%%%%%
\section{Definition of the problem}

Let $\mbf{X}(t) = (X_1(t), \dots, X_N(t))^T$ be a multivariate, stationary stochastic process whose realizations $\mbf{x}(t)$ cannot be observed directly; indirect information on $\mbf{x}(t)$ can be obtained by observing the realizations $\mbf{y}(t)$ of the process $\mbf{Y}(t)$, which is a noisy linear mixture of $\mbf{X}(t)$
\begin{equation}\label{eq:model_eq}
\mbf{Y}(t) = \mbf{G} \mbf{X}(t) + \mbf{N}(t)
\end{equation}
where $\mbf{G}$ is an $M \times N$ \textit{forward} matrix, with $M = \dim (\mbf{Y}(t))$, $N = \dim (\mbf{X}(t))$, and $\mbf{N}(t)$ is the measurement noise process, which is assumed to be a zero--mean Gaussian process independent from $\mbf{X}(t)$. For ease of presentation, we  further assume $M \leq N$ and $\mbf{G}$ to be a full-rank matrix so that  all its singular values are strictly positive; however, the results below can be easily extended to the general case.

We consider the case where one is interested in reconstructing the cross--spectrum of the process $\mbf{X}(t)$, that contains information on the statistical dependencies between the different components of the signal. The cross--spectrum is a one--parameter family of $N \times N$ matrices $\mbf{S}^{\mbf{X}}(f)$, whose $(j,k)$--th element is defined as 
\begin{equation} \label{eq:def_SX}
S^\mbf{X}_{j,k}(f)=\lim_{T\rightarrow+\infty}\frac{1}{T} E [\hat{X}_j(f,T) \hat{X}_k(f,T)^*]
\end{equation}
%S^X_{i,j}(f)=\lim_{T\rightarrow+\infty}\frac{1}{T} E [Z_k(f) Z_j(f)^*],
where $\hat{X}_j(f,T)$ is the Fourier transform of $X_j(t)$ over the interval $[0,T]$, defined as 
\begin{equation} 
\hat{X}_j(f,T)=\int_{0}^{T} X_j(t)e^{-2\pi ift}dt 
\end{equation}
and $X^*$ is the complex conjugate of $X$ \cite{be_pi11}.

In this work we consider the case when the reconstruction of the cross--spectrum is done in a two--step process: 
\begin{itemize}
\item[(i)] First, a regularized estimate $\mbf{x}_{\lambda}(t)$ of $\mbf{x}(t)$ is computed as
%\begin{equation}\label{eq:filter_factors}
%     \mbf{x}_{\Phi}(t) = \mbf{W}_{\lambda} \mbf{y}(t) = \mbf{V} \mbf{\Phi} \Sigma^{^{\dagger}} \mbf{U}^t \mbf{y}(t) ~~~,
%\end{equation}
\begin{equation}\label{eq:filter_factors}
     \mbf{x}_{\lambda}(t) = \mbf{W}_{\mbf{\lambda}} \mbf{y}(t) = \mbf{V} \mbf{\Phi}(\lambda)\mbf{\Sigma}^{^{\dagger}} \mbf{U}^t \mbf{y}(t) ~~~,
\end{equation}
where $\mbf{G} = \mbf{U} \mbf{\Sigma} \mbf{V}^t $ is the singular value decomposition (SVD) of the forward matrix; $\mbf{\Sigma}^{^{\dagger}}$ is the pseudo-inverse of $\mbf{\Sigma}$ and $\mbf{\Phi}(\lambda) = \textup{diag}( \varphi_1(\lambda), \dots, \varphi_{M}(\lambda) ) \in R^{N \times N}$ are the \textit{filter factors} \cite{hansen05}, which are functions of one (or more) \textit{regularization parameter(s)} $\lambda$. 

\item[(ii)]
Then, an estimate of the cross--spectrum is obtained from these reconstructed time--series using the Welch's method \cite{welch1967use}, which consists in partitioning the data in $P$ overlapping segments $\{\mbf{x}_{\lambda}^p(t)\}_{p=1,\dots,P}$, computing $\hat{\mbf{x}}_{\lambda}^p(f)$, the Discrete Fourier Transform of the signals multiplied by a window function $w(t)$, and then averaging these modified periodograms:

\begin{equation}\label{eq:welch}
    \mbf{S}^{\mbf{x}_{\lambda}}(f)=\frac{L}{PW}\sum_{p=1}^{P} \hat{\mbf{x}}_{\lambda}^p(f) \hat{\mbf{x}}_{\lambda}^p(f)^*, \hspace{1cm} f=0, \dots, L-1, 
\end{equation}
where $L$ is the length of each segment and $W=\frac{1}{L}\sum_{t=0}^{L-1}w(t)^2$.
\end{itemize}

This two--step approach is largely used, e.g., in connectivity estimation from M/EEG data, where the estimated cross-spectrum is typically used to compute a large pool of connectivity metrics such as coherence \cite{nunez99}, imaginary part of coherency \cite{nolte04}, phase slope index \cite{nolte08}. Naturally but crucially, this estimate depends on the choice of the regularization method, as well as on the choice of $\lambda$, which modulates the degree of regularization of the estimate $\mbf{x}_{\mbf{\lambda}}$.

In this work we will mainly focus on two regularization methods, namely truncated SVD (tSVD) and Tikhonov regularization. The reason of this specific choice is as follows: the Tikhonov method is one of the more commonly employed methods for connectivity estimation in M/EEG, and it has been used by Hincapi\'e and colleagues in the paper that motivated this study \cite{hietal16}; tSVD is a method which is easy to deal with analytically, and in addition it will provide a different result than the Tikhonov method, thus showing that the answer to the main question of this study is method--dependent. 

Henceforth, $\mbf{u}_i$ and $\mbf{v}_i$ will denote the $i$-th column of matrices $\mbf{U}$ and $\mathbf{V}$, respectively. tSVD relies on the 1--parameter family of regularized estimates
\begin{equation}
	\label{eq:tsvd_estimate}
    \mbf{x}_{\lambda}(t) = \sum_{i=1}^{\lambda} \frac{\mbf{u}_i^t \mbf{y}(t)}{\sigma_i}\mbf{v}_i~~~~~ \lambda \in \left\{1, \dots M \right\} ~~~,
\end{equation}
which are obtained from equation (\ref{eq:filter_factors}) by setting 
\begin{equation}
    \varphi_i(\lambda) = \left\{
    \begin{array}{ccc}
    1 & \textrm{if} & i \leq \lambda \\
    0 & \textrm{if} & i > \lambda \\
    \end{array} \right. ~~~ .
\end{equation}
Tikhonov estimates are defined as
\begin{equation}
	\label{eq:tikhonov_estimate}
    \mbf{x}_{\lambda}(t) = \sum_{i=1}^{M} \frac{\sigma_i^2}{\sigma_i^2 + \lambda} \frac{\mbf{u}_i^t \mbf{y}(t)}{\sigma_i}\mbf{v}_i~~~~~ \lambda \geq 0 ~~~,
\end{equation} 
which are obtained from equation (\ref{eq:filter_factors}) by setting $\varphi_i(\lambda) = \frac{\sigma_i^2}{\sigma_i^2+\lambda}$. From now on, for semplicity, we omit the dependence of $\mbf{\Phi}$ and $\varphi_i$ from $\lambda$. Also note that in the two methods the parameter $\lambda$ assumes values in different sets. In tSVD $\lambda$ determines the number of retained SVD components, and therefore assumes integer values in $\{1,\dots,M\}$, where a small $\lambda$ value means few retained components and thus an high level of regularization. In Tikhonov regularization $\lambda$ determines the strength with which each SVD component contributes to the solution; in this case $\lambda$ assumes continuous values in $[0,+\infty)$ and the higher the value the higher the degree of regularization. 

For the two mentioned methods, we consider the problem of the optimal choice of the regularization parameter $\lambda$ for reconstruction of the cross--spectrum. We define optimality through the minimization of the norm of the discrepancy, specifically we define the two following optimal values for the parameter.
\begin{definition}\label{def:opt_values}
Consider the regularized solution (\ref{eq:filter_factors}) and the cross--spectrum estimate (\ref{eq:welch})  associated to a realization of equation (\ref{eq:model_eq}); we define the optimal parameter for the reconstruction of $\mbf{x}(t)$ 
\begin{equation} \label{eq:err_x}
    \lambda_{\mbf{x}}^* = \arg \min_{\lambda}\varepsilon_{\mbf{x}}(\lambda) ~~~ \textrm{with} ~~~  \varepsilon_{\mbf{x}}(\lambda) = \sum_t \left\| \mbf{x}_{\lambda}(t) - \mbf{x}(t) \right\|_2^2 ~~~,
\end{equation}
and the optimal parameter for the reconstruction of $\mbf{S}^{\mbf{x}}(f)$
\begin{equation}\label{eq:err_S_x}
    \lambda_{\mbf{S}}^* = \arg \min_{\lambda} \varepsilon_{\mbf{S}}(\lambda) ~~~ \textrm{with} ~~~ \varepsilon_{\mathbf{S}}(\lambda) = \sum_f \left\| \mbf{S}^{\mbf{x}_{\lambda}}(f) - \mbf{S}^{\mbf{x}}(f) \right\|_F^2~~~,
\end{equation}
where $\|\cdot\|_2$ and $\|\cdot\|_F$ are the $L^2$-norm and the Frobenius norm, respectively; $\varepsilon_{\mbf{x}}(\lambda)$ and $\varepsilon_{\mbf{S}}(\lambda)$ will be called reconstruction errors.
\end{definition}

In the following sections we shall answer the following question: does the optimal regularization parameter for reconstruction of $\mbf{x}(t)$, $\lambda_{\mbf{x}}^*$, coincide with the optimal regularization parameter for reconstruction of $\mbf{S}^\mbf{x}(f)$, $\lambda_{\mbf{S}}^*$?

%%%%%%%%%%%%%%%%%%%%%%%%%%%%%%%%%%%%%%%%%%%%%%%%%%%%%%%%%%%%%%%%%%%%%%%%%%%%%%%%%%%%%%%
%\section{Resolution matrices and filter factors analysis of $\varepsilon_{\mbf{x}}(\lambda)$ and $\varepsilon_{\mbf{S}}(\lambda)$}
\section{Reconstruction errors with filter factors}
In this section we aim at deriving an explicit formulation of $\varepsilon_{\mbf{x}}(\lambda)$ and $\varepsilon_{\mbf{S}}(\lambda)$ in terms of the filter factors $\mbf{\Phi}$. To this end we observe that from equations (\ref{eq:model_eq}) and (\ref{eq:filter_factors}) we can derive the following relationship between the true and the reconstructed signal:
\begin{equation}\label{discr_x}
    \mbf{x}_{\lambda}(t) = \mbf{R}_{\lambda} \mbf{x}(t) + \mbf{W}_{\lambda} \mbf{n}(t)
\end{equation}
where $\mbf{R}_{\lambda} = \mbf{W}_{\lambda} \mbf{G}$ is the resolution matrix \cite{dePeMe96, hansen05}.

A similar relationship between the true and the estimated cross-spectrum can be derived by substituting equation (\ref{discr_x}) into definition (\ref{eq:welch}) and by exploiting the linearity of the Discrete Fourier Transform:
\begin{eqnarray}\label{eq:discr_S}
\eqalign{
    \mbf{\mathcal{S}}^{\mbf{x}_{\lambda}}(f) & = \left( \mbf{R}_{\lambda} \otimes  \mbf{R}_{\lambda} \right) \mbf{\mathcal{S}}^{\mbf{x}}(f) + 
    \left( \mbf{W}_{\lambda} \otimes \mbf{W}_{\lambda} \right) \mbf{\mathcal{S}}^{\mbf{n}}(f)\\
    & =  \left( \mbf{W}_{\lambda} \otimes \mbf{R}_{\lambda} \right) \mbf{\mathcal{S}}^{\mbf{x}\mbf{n}}(f) + \left( \mbf{R}_{\lambda} \otimes \mbf{W}_{\lambda} \right) \mbf{\mathcal{S}}^{\mbf{n}\mbf{x}}(f)} ~~~,
\end{eqnarray}
where $\mbf{\mathcal{S}}^{\mbf{x}}(f)$ is the vector obtained by concatenating the columns of the matrix $\mbf{S}^{\mbf{x}}(f)$,  $\otimes$ is the Kronecker product, and $\mbf{S}^{\mbf{x}\mbf{n}}$ is the cross--spectrum between $\mbf{x}$ and $\mbf{n}$, i.e., following the notation in equation (\ref{eq:welch}),  $\mbf{S}^{\mbf{x}\mbf{n}}(f)=\frac{L}{PW}\sum_{p=1}^{P} \hat{\mbf{x}}^p(f) \hat{\mbf{n}}^p(f)^*$. 

Since $\mbf{X}(t)$ and $\mbf{N}(t)$ are independent, $\mbf{\mathcal{S}}^{\mbf{x}\mbf{n}}(f)$ and $\mbf{\mathcal{S}}^{\mbf{n}\mbf{x}}(f)$ are negligible provided that enough data time-points are available. Hence from definition \ref{def:opt_values} it follows
\begin{equation}\label{eq:eps_x_rm}
    \varepsilon_{\mbf{x}}(\lambda) = \sum_t \left\| \left( \mbf{R}_{\lambda} - \mbf{I}_N \right) \mbf{x}(t) + \mbf{W}_{\lambda} \mbf{n}(t) \right\|_2^2
\end{equation}
\begin{equation}\label{eq:eps_S_rm}
    \varepsilon_{\mathbf{S}}(\lambda) = \sum_f \left\| \left( \mbf{R}_{\lambda} \otimes  \mbf{R}_{\lambda} - \mbf{I}_{N^2}\right) \mbf{\mathcal{S}}^{\mbf{x}}(f) + 
    \left( \mbf{W}_{\lambda} \otimes \mbf{W}_{\lambda} \right) \mbf{\mathcal{S}}^{\mbf{n}}(f) \right\|_2^2
\end{equation}
where $\mbf{I}_N$ is the identity matrix of size $N \times N$.
\\

\begin{proposition}\label{prop:errors}
The reconstruction errors defined in (\ref{eq:err_x}) and (\ref{eq:err_S_x}) are given by:
\begin{eqnarray}
\fl
    \varepsilon_{\mbf{x}}(\lambda) = \sum_t \sum_{i=M+1}^N \left(\mbf{v}_i^t \mbf{x}(t)\right)^2 + \sum_t \sum_{i=1}^M \left[ \left( \varphi_i - 1 \right)^2 \left(\mbf{v}_i^t \mbf{x}(t)\right)^2 + \varphi_i^2 \frac{\left(\mbf{u}_i^t \mbf{n}(t) \right)^2}{\sigma_i^2}  \right]
    \label{eq:eps_x_filt}
\end{eqnarray}
and
\begin{eqnarray}
\fl \eqalign{
\varepsilon_{\mathbf{S}}(\lambda) & = 
\sum_f \sum_{i\geq M+1\ or \atop j\geq M+1} \left| (\mbf{v}_i \otimes \mbf{v}_j)^t \mathcal{S}^{\mbf{x}}(f) \right|^2 
+ \sum_f \sum_{i,j=1}^M \Bigg[ \left(\varphi_i \varphi_j - 1 \right)^2 \left| (\mbf{v}_i \otimes \mbf{v}_j)^t \mathcal{S}^{\mbf{x}}(f) \right|^2 \\
& + \left(\frac{\varphi_i \varphi_j}{\sigma_i \sigma_j}\right)^2 \left|(\mbf{u}_i \otimes \mbf{u}_j)^t \mathcal{S}^{\mbf{n}}(f)\right|^2 
+ 2 \left(\varphi_i \varphi_j - 1 \right) \frac{\varphi_i \varphi_j}{\sigma_i \sigma_j} Re\left(\overline{(\mbf{v}_i \otimes \mbf{v}_j)^t \mathcal{S}^{\mbf{x}}(f)} (\mbf{u}_i \otimes \mbf{u}_j)^t \mathcal{S}^{\mbf{n}}(f) \right)\Bigg]}\label{eq:eps_S_filt}
\end{eqnarray}
\end{proposition}

\begin{proof}
To prove equation (\ref{eq:eps_x_filt}) we observe that
\begin{displaymath}
\mbf{W}_{\lambda} = \mbf{V} \mbf{\Phi} \mbf{\Sigma}^{^{\dagger}} \mbf{U}^t = \sum_{i=1}^M \mbf{v}_i \frac{\varphi_i}{\sigma_i} \mbf{u}_i^t
\end{displaymath}
and
\begin{displaymath}
\mbf{R}_{\lambda} - \mbf{I}_N = \mbf{V} \mbf{\Phi} \mbf{\Sigma}^{^{\dagger}}\mbf{\Sigma} \mbf{V}^t - \mbf{I}_N = \sum_{i=1}^M \mbf{v}_i (\varphi_i -1 ) \mbf{v}_i^t - \sum_{i=M+1}^N \mbf{v}_i \mbf{v}_i^t.
\end{displaymath}
Then the thesis follows from equation (\ref{eq:eps_x_rm}) by exploiting the orthonormality of $\mbf{V}$ and the independence between processes $ \mbf{X}(t)$ and $\mbf{N}(t) $.\\
Analogously, equation (\ref{eq:eps_S_filt}) follows from equation (\ref{eq:eps_S_rm}) by observing
\begin{displaymath}
\fl 
\mbf{W}_{\lambda} \otimes \mbf{W}_{\lambda} = \left( \mbf{V} \otimes \mbf{V} \right) \left( \mbf{\Phi} \mbf{\Sigma}^{^{\dagger}} \otimes \mbf{\Phi} \mbf{\Sigma}^{^{\dagger}} \right) \left(\mbf{U} \otimes \mbf{U} \right)^t = \sum_{i,j=1}^M (\mbf{v}_i \otimes \mbf{v}_j) \frac{\varphi_i \varphi_j}{\sigma_i \sigma_j} (\mbf{u}_i \otimes \mbf{u}_j )^t 
\end{displaymath}
and
\begin{displaymath}
\fl 
\eqalign{
\mbf{R}_{\lambda} \otimes \mbf{R}_{\lambda} - \mbf{I}_{N^2} = \left( \mbf{V} \otimes \mbf{V} \right) \left( \mbf{\Phi} \mbf{\Sigma}^{^{\dagger}}\mbf{\Sigma} \otimes \mbf{\Phi} \mbf{\Sigma}^{^{\dagger}}\mbf{\Sigma}\mbf{}- \mbf{I}_{N^2} \right) \left(\mbf{V} \otimes \mbf{V} \right)^t = \\
\sum_{i,j=1}^M (\mbf{v}_i \otimes \mbf{v}_j) (\varphi_i \varphi_j-1) (\mbf{v}_i \otimes \mbf{v}_j )^t - \sum_{i\geq M+1\ or \atop j\geq M+1} (\mbf{v}_i \otimes \mbf{v}_j) (\mbf{v}_i \otimes \mbf{v}_j )^t} 
\end{displaymath}

\end{proof}

\begin{remark}
The expression in (\ref{eq:eps_x_filt}) is a classical result in regularization theory \cite{hansen05}, in which the reconstruction error is expressed in terms of three distinct components. The first component is the norm of the projection of the original signal onto the kernel of $\mbf{G}$, i.e. the part of the signal that cannot be reconstructed. The second term is the \textit{regularization error}, i.e. the error introduced by regularization itself; indeed, this term vanishes when the value of all the filters is one. The last term is the \textit{perturbation error}, i.e. the backprojection of stochastic noise components onto the reconstructed signal, that regularization tries to reduce.
\end{remark}

\begin{remark} 
Expression (\ref{eq:eps_S_filt}) is the analogue of (\ref{eq:eps_x_filt}) for the cross--spectrum estimated with the two--step approach. To the best of our knowledge this expression is novel and has never been studied. The reconstruction error $\varepsilon_{\mathbf{S}}(\lambda)$ here is made of four distinct components: three of them have the same interpretation of those appearing in $\varepsilon_{\mathbf{x}}(\lambda)$; the fourth term is a non--vanishing mixed term, that depends on both the signal and the noise spectra; as we shall see below, this term turns out to be negative at least in some special cases. 
\end{remark}

\section{The relationship between the optimal regularization parameters: two case studies}

We will now address the main question posed in the introduction: does the optimal regularization parameter for reconstruction of the time--series coincide with the optimal regularization parameter for reconstruction of the cross--spectrum?
As we shall see, the answer depends on the specific choice of the inverse method, i.e. of the form of the filter factors. Here we study first the case of tSVD, and then the case of the Tikhonov method.

In order to proceed analytically, in this section we make the further assumption that both the signal and the noise are white--noise Gaussian processes, with covariance matrices $\omega^2 \mbf{I}_M$ and $\alpha^2 \mbf{I}_N$, respectively. The Gaussian assumption is often not too far fetched; in M/EEG, particularly, it is widely used and, even though perhaps the data distribution is not exactly Gaussian, the Gaussian assumption is implicit (when not explicit) in the vast majority of connectivity studies \cite{nolte2019mathematical}. The white--noise assumption, on the other hand, is stronger, as it implies that there is no temporal structure in the signal: we will come back to this point in the Discussion.

\subsection{Truncated SVD}
When tSVD is employed, by substituting the values of the corresponding filter factors into equations (\ref{eq:eps_x_filt}) and (\ref{eq:eps_S_filt}) we get the following corollary of Proposition \ref{prop:errors}.

\begin{corollary}
Consider the tSVD estimate $\mbf{x}_{\lambda}(t)$ given by equation (\ref{eq:tsvd_estimate}), with regularization parameter
$\lambda \in \left\{1, \dots, M \right\}$. Then
\begin{equation}\label{eq:eps_x_tsvd}
    \varepsilon_{\mbf{x}}(\lambda) = \sum_t \sum_{i=\lambda+1}^{N} \left(\mbf{v}_i^t \mbf{x}(t)\right)^2 + \sum_t \sum_{i=1}^{\lambda} \frac{\left(\mbf{u}_i^t \mbf{n}(t) \right)^2}{\sigma_i^2} 
\end{equation}
and
\begin{equation}\label{eq:eps_S_tsvd}
\fl
     \varepsilon_{\mathbf{S}}(\lambda) =
    \sum_f \sum_{i\geq \lambda+1\ or \atop j\geq \lambda+1} \left| (\mbf{v}_i \otimes \mbf{v}_j)^t \mathcal{S}^{\mbf{x}}(f) \right|^2  + \sum_f \sum_{i, j= 1}^{\lambda}  \frac{\left|(\mbf{u}_i \otimes \mbf{u}_j)^t \mathcal{S}^{\mbf{n}}(f)\right|^2}{\sigma_i^2 \sigma_j^2}
\end{equation}
\end{corollary}

\begin{remark}
When regularization is accomplished through tSVD, the mixed term in $\varepsilon_{\mathbf{S}}(\lambda)$ vanishes; this allows us to compute the optimal regularization parameter explicitly.
\end{remark}

\begin{theorem}\label{theo:tsvd}
Let $x_{\lambda}(t)$ be the tSVD estimate as given by equation (\ref{eq:tsvd_estimate}), with regularization parameter $\lambda \in \left\{1, \dots, M \right\}$; assume $\mbf{X}(t)$ and $\mbf{N}(t)$ to be white noise processes with covariance matrices $\omega^2 \mbf{I}_N$ and $\alpha^2 \mbf{I}_{M}$, respectively. Then
\begin{equation}
    \lambda_{\mbf{x}}^* = \lambda_{\mbf{S}}^* = \max \left\{ \lambda \in \{1, \dots, M  \} \ s.t \  \sigma_{\lambda} \geq \frac{\alpha}{\omega} \right\}
\end{equation}
\end{theorem}

\begin{proof}
Provided that enough data time-points are available, the result follows from the assumptions on $\mbf{X}(t)$ and $\mbf{N}(t)$ which ensure
\begin{displaymath}
\sum_t \left(\mbf{v}_i^t \mbf{x}(t)\right)^2 = T \omega^2 \quad \quad \sum_{t} \left(\mbf{u}_i^t \mbf{n}(t) \right)^2 = T \alpha^2
\end{displaymath}
and

\begin{displaymath}
\left| (\mbf{v}_i \otimes \mbf{v}_j)^t \mathcal{S}^{\mbf{x}}(f) \right|^2 = \omega^4 \delta_{ij} \quad \quad \left|(\mbf{u}_i \otimes \mbf{u}_j)^t \mathcal{S}^{\mbf{n}}(f)\right|^2 = \alpha^4 \delta_{ij}
\end{displaymath}
where $\delta_{ij}$ is the Kronecker delta.\\

Such approximations allow to further simplify equations (\ref{eq:eps_x_tsvd}) and (\ref{eq:eps_S_tsvd}) that now read
\begin{displaymath}
\varepsilon_{\mbf{x}}(\lambda) = \left( N - \lambda \right) T \omega^2 + T \alpha^2 \sum_{i=1}^{\lambda} \frac{1}{\sigma^2_i}
\end{displaymath}
\begin{displaymath}
\varepsilon_{\mathbf{S}}(\lambda) = \left( N - \lambda \right) L \omega^4 + L \alpha^4 \sum_{i=1}^{\lambda} \frac{1}{\sigma_i^4}
\end{displaymath}
The thesis follows by observing that the increments 
\begin{displaymath}
\varepsilon_{\mbf{x}}(\lambda) - \varepsilon_{\mbf{x}}(\lambda - 1) = - T \omega^2 + \frac{T \alpha^2}{\sigma_{\lambda}^{2}}   
\end{displaymath}
and
\begin{displaymath}
\varepsilon_{\mbf{S}}(\lambda) - \varepsilon_{\mbf{S}}(\lambda - 1) = - L \omega^4 + \frac{L \alpha^4}{\sigma_{\lambda}^{4}} 
\end{displaymath}
are non--decreasing functions of $\lambda$ and thus $\varepsilon_{\mbf{x}}(\lambda)$ and $\varepsilon_{\mbf{S}}(\lambda)$ have a unique minimum at the biggest $\lambda$ for which such increments are negative.
\end{proof}

\subsection{Tikhonov}

We now consider the case when regularization is performed by means of the standard Tikhonov formula.

\begin{corollary}
Let $x_{\lambda}(t)$ be the Tikhonov estimate as given by equation (\ref{eq:tikhonov_estimate}), with regularization parameter $\lambda \geq 0$. Then
\begin{eqnarray}
\fl
\eqalign{
    \varepsilon_{\mbf{x}}(\lambda) & = \sum_t \sum_{i=M+1}^N \left(\mbf{v}_i^t \mbf{x}(t)\right)^2 \\
    & + \sum_t \sum_{i=1}^M \left[ \frac{\lambda^2}{(\sigma^2_i + \lambda)^2} \left(\mbf{v}_i^t \mbf{x}(t)\right)^2 + \frac{\sigma^2_i}{(\sigma^2_i + \lambda)^2} \left(\mbf{u}_i^t \mbf{n}(t) \right)^2  \right] }\label{eq:eps_x_tik}
\end{eqnarray}
and
\begin{eqnarray}
\fl \eqalign{
\varepsilon_{\mathbf{S}}(\lambda) & = \sum_f \sum_{i\geq M+1\ or \atop j\geq M+1} \left| (\mbf{v}_i \otimes \mbf{v}_j)^t \mathcal{S}^{\mbf{x}}(f) \right|^2 + \\
& \sum_f \sum_{i,j=1}^M \Big[ \left(\frac{\sigma^2_i\sigma^2_j}{(\sigma_i^2+\lambda)(\sigma_j^2+\lambda)} - 1 \right)^2 \left| (\mbf{v}_i \otimes \mbf{v}_j)^t \mathcal{S}^{\mbf{x}}(f) \right|^2 \\
& + \frac{\sigma^2_i\sigma^2_j}{(\sigma_i^2+\lambda)^2(\sigma_j^2+\lambda)^2} \left|(\mbf{u}_i \otimes \mbf{u}_j)^t \mathcal{S}^{\mbf{n}}(f)\right|^2   \\
& + 2 \left(\frac{\sigma^2_i\sigma^2_j}{(\sigma_i^2+\lambda)(\sigma_j^2+\lambda)} - 1 \right) \frac{\sigma_i\sigma_j}{(\sigma^2_i+\lambda)(\sigma^2_j+\lambda)} Re\left(\overline{(\mbf{v}_i \otimes \mbf{v}_j)^t \mathcal{S}^{\mbf{x}}(f)} (\mbf{u}_i \otimes \mbf{u}_j)^t \mathcal{S}^{\mbf{n}}(f)\right) \Big]}\label{eq:eps_S_tik}
\end{eqnarray}
\end{corollary}

Again we assume that $\mbf{X}(t)$ and $\mbf{N}(t)$ are white--noise Gaussian processes with covariance matrices $\omega^2 \mbf{I}_N$ and $\alpha^2 \mbf{I}_{M}$. Under this assumption equations (\ref{eq:eps_x_tik}) and (\ref{eq:eps_S_tik}) become
\begin{eqnarray}
\fl \label{eq:eps_x_tik_2}
  \varepsilon_{\mbf{x}}(\lambda) = T (N-M) \omega^2 + T \omega^2 \sum_{i=1}^M \frac{\lambda^2}{(\sigma^2_i + \lambda)^2} + T \alpha^2 \sum_{i=1}^M \frac{\sigma^2_i}{(\sigma^2_i + \lambda)^2}   
\end{eqnarray}
and
\begin{eqnarray}
\fl \eqalign{ \label{eq:eps_S_tik_2}
\varepsilon_{\mathbf{S}}(\lambda) & = L (N-M) \omega^4 + L \omega^4 \sum_{i=1}^M \left(\frac{\sigma^4_i}{(\sigma_i^2+\lambda)^2} - 1 \right)^2 + L \alpha^4 \sum_{i=1}^M \frac{\sigma^4_i}{(\sigma_i^2+\lambda)^4}\\
& + 2 L \omega^2 \alpha^2 \sum_{i=1}^M \left(\frac{\sigma^4_i}{(\sigma_i^2+\lambda)^2} - 1 \right) \frac{\sigma^2_i}{(\sigma^2_i+\lambda)^2}} ~~~, 
\end{eqnarray}
where we notice that, as anticipated in the previous section, the fourth addend is negative; this fact suggests that, to the extent that the other terms are comparable to those in the corresponding expression for the tSVD (\ref{eq:eps_S_tsvd}), the reconstruction error generated by the Tikhonov method is smaller than the one generated by tSVD.

By differentiating equations (\ref{eq:eps_x_tik_2}) and (\ref{eq:eps_S_tik_2}) we have
\begin{eqnarray}
\fl
  \frac{d}{d\lambda}\varepsilon_{\mbf{x}}(\lambda) = 2T \left(\omega^2 \lambda - \alpha^2 \right)  \sum_{i=1}^M \frac{\sigma_i^2}{(\sigma^2_i + \lambda)^3}\label{eq:de_eps_x_tik}   
\end{eqnarray}
and
\begin{eqnarray}
\fl \eqalign{
 \frac{d}{d\lambda}\varepsilon_{\mbf{S}}(\lambda) =  4 L \omega^2 \sum_{i=1}^M & \frac{\sigma_i^2}{(\sigma_i^2+\lambda)^5} (\alpha^2 + \sigma_i^2 \omega^2)\cdot \\
 & \cdot\left(\lambda + \sigma_i^2 + \sqrt{\sigma_i^4 + \sigma_i^2 \frac{\alpha^2}{\omega^{2}}} \right) \left(\lambda + \sigma_i^2 - \sqrt{\sigma_i^4 + \sigma_i^2 \frac{\alpha^2}{\omega^{2}}} \right)}~~~. \label{eq:de_eps_S_tik}
\end{eqnarray}

We are now able to prove the following theorem.
\begin{theorem}\label{theo:tik}
Let $x_{\lambda}(t)$ be the Tikhonov estimate as given by equation (\ref{eq:tikhonov_estimate}), with regularization parameter $\lambda \geq 0$; assume $\mbf{X}(t)$ and $\mbf{N}(t)$ to be white--noise Gaussian processes with covariance matrices $\omega^2 \mbf{I}_N$ and $\alpha^2 \mbf{I}_{M}$, respectively. Then
\begin{equation}\label{eq:lam_x_tik}
    \lambda_{\mbf{x}}^* = \frac{\alpha^2}{\omega^2} 
\end{equation}
and
\begin{equation}\label{eq:lam_S_tik}
    \lambda_{\mbf{S}}^*\ <\ \frac{\lambda_\mbf{x}^*}{2}
\end{equation}
\end{theorem}
\begin{proof}
The first statement simply follows from equation (\ref{eq:de_eps_x_tik}) by observing that $\frac{d}{d\lambda}\varepsilon_{\mbf{x}}(\lambda) \geq 0$ if and only if $\lambda \geq \frac{\alpha^2}{\omega^2}$. \\
Instead, equation (\ref{eq:de_eps_S_tik}) implies that $\frac{d}{d\lambda}\varepsilon_{\mbf{S}}(\lambda) > 0$ if 
\begin{equation}\label{eq:condition}
    \lambda\ > - \sigma_i^2 + \sqrt{\sigma_i^4 + \sigma_i^2 \frac{\alpha^2}{\omega^{2}}} ~~~ .
\end{equation}
Consider the function $h: [0, +\infty) \ni z \rightarrow -z^2 + \sqrt{z^4 + z^2 \frac{\alpha^2}{\omega^2}}$. As schematically shown in Figure \ref{fig:h}, $h$ is strictly increasing and bounded above by $\frac{\lambda_\mbf{x}^*}{2} = \frac{\alpha^2}{2\omega^2}.$ As a consequence, the condition (\ref{eq:condition}) is satisfied if $\lambda \geq \frac{\lambda_\mbf{x}^*}{2}$, that means $\varepsilon_{\mbf{S}}(\lambda)$ is strictly increasing in $[\frac{\lambda_\mbf{x}^*}{2}, +\infty)$ and thus inequality (\ref{eq:lam_S_tik}) holds.
\end{proof}

The main interest of Theorem \ref{theo:tik} is that it provides a simple relationship between $\lambda_{\mbf{S}}^*$ and $\lambda_{\mbf{x}}^*$.
However, expression (\ref{eq:de_eps_S_tik}) contains more information about the values of $\lambda_{\mbf{S}}^*$, as stated in the following Theorem.
\begin{figure}
    \centering
    \begin{tikzpicture}[domain=0:9]
      \fill [green, opacity=0.3] (0,0) rectangle ( 9,{-1+sqrt(1+12) });
 	\fill [blue, opacity=0.2] (0,{-4^2+sqrt(4^4+12*4^2) } ) rectangle (9,6.8);
	\draw[ ->] (-0.2,0) -- (9.2,0) node[right] {$z$};
      \draw[->] (0,-0.2) -- (0,7) node[above] {$h(z)$};
      \draw[color=black] plot (\x,{ -\x^2+sqrt(\x^4+12*\x^2) } ) node[right] {$h(z)$};
      \draw[color=black] plot (\x,{6});
      \node [align=left] at (-0.3,6) {$\frac{\lambda_\mbf{x}^*}{2}$}; 
      \draw[-,dashed] (1,0) -- (1, {-1+sqrt(1+12) } ) node  at (1,-0.4) {$\sigma_M$};
      \draw[-,dashed] (4,0) -- (4, {-4^2+sqrt(4^4+12*4^2) } ) node at (4,-0.4) {$\sigma_1$}; 
      \draw[-,dashed] (0,{-1+sqrt(1+12) }) -- (1, {-1+sqrt(1+12) } ) node  at (-0.6,{-1+sqrt(1+12) }) {$h(\sigma_M)$};
      \draw[-,dashed] (0,{-4^2+sqrt(4^4+12*4^2) } ) -- (4, {-4^2+sqrt(4^4+12*4^2) } ) node at (-0.6,{-4^2+sqrt(4^4+12*4^2) } ) {$h(\sigma_1)$}; 
	%\draw [ultra thick] (0,{-1+sqrt(1+12) })--(0,{-4^2+sqrt(4^4+12*4^2) } ) ;
	\node  [dgreen] at (7,1.3) {$\frac{d}{d\lambda}\varepsilon_\mbf{S}(\lambda)<0$};
	\node  [blue] at (7,6.4) {$\frac{d}{d\lambda}\varepsilon_\mbf{S}(\lambda)>0$};
	\node at (1,0) {$|$};
	\node at (2.5,0) {$|$};
	\node at (2.5,-0.4) {$\sigma_3$};
	\node at (3.5,0) {$|$};
    \node at (3.5,-0.4){$\sigma_2$};
	\node at (4,0) {$|$};
	\node at (1.75,-0.4) {$\cdots$};
	\draw [|-] (-1.3,{-4^2+sqrt(4^4+12*4^2) })--(-1.3,4.28);
	\draw [|-] (-1.3,{-1+sqrt(1+12) })--(-1.3,3.5);
	\node at (-1.3,3.89) {$\lambda_\mbf{S}^*$\;?};
\end{tikzpicture}
    \caption{Plot of the function $h(z)$ defined in Theorem \ref{theo:tik}. If $\lambda<h(\sigma_M)$ $\varepsilon_\mbf{S}(\lambda)$ is decreasing (green area), whereas if $\lambda>h(\sigma_1)$ $\varepsilon_\mbf{S}(\lambda)$ is increasing (blue area); therefore the optimal regularization parameter $\lambda_\mbf{S}^*$ lies in the interval $[h(\sigma_M),h(\sigma_1)]$. Moreover, for $\lambda\geq\frac{\lambda_\mbf{x}^*}{2}$ $\varepsilon_\mbf{S}(\lambda)$ is increasing independently from the singular values; this fact leads to the inequality $\lambda_\mbf{S}^*<\frac{\lambda_\mbf{x}^*}{2}$. }
    \label{fig:h}
\end{figure}
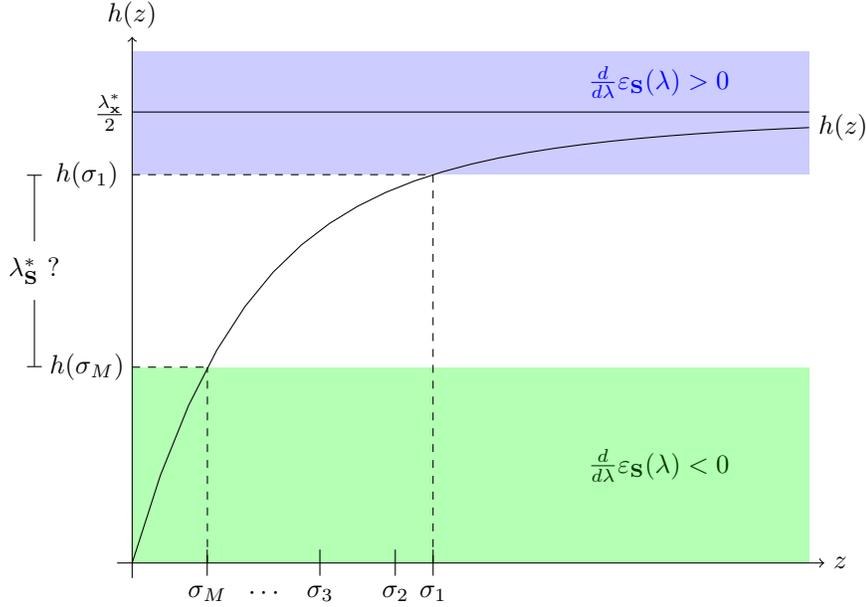
\begin{theorem} \label{theo:interval}
Under the same hypotheses of Theorem \ref{theo:tik}, the value of $\lambda_{\mbf{S}}^*$ belongs to the interval $[h(\sigma_M), h(\sigma_1)]$, where $h(z) = -z^2 + \sqrt{z^4 + z^2 \; \frac{\alpha^2}{\omega^2}}$.
\end{theorem}
\begin{proof}
As schematically shown in Figure \ref{fig:h}, when $\lambda>h(\sigma_1)$, all the addends in (\ref{eq:de_eps_S_tik}) are positive and thus $\frac{d}{d\lambda}\varepsilon_{\mbf{x}}(\lambda)$ is positive; on the other hand, when $\lambda<h(\sigma_M)$ the derivative  $\frac{d}{d\lambda}\varepsilon_{\mbf{x}}(\lambda)$ is negative as all the addends are negative.
\end{proof}

\begin{remark}
Theorem \ref{theo:interval} also gives information on the limiting behaviour of $\lambda_{\mbf{S}}^*$ as $\lambda_{\mbf{x}}^* = \frac{\alpha^2}{\omega^2}$ approaches very small or very large values. In the no--noise scenario, when $\lambda_{\mbf{x}}^* \sim 0$, $\lambda_{\mbf{S}}^*$ grows approximately linearly with $\lambda_{\mbf{x}}^*$. The other boundary is however more interesting. Indeed, when $\lambda_{\mbf{x}}^* \rightarrow \infty$ the extremes of the interval $h(\sigma_1)$ and $h(\sigma_M)$ grow with the same order of $\sqrt{\lambda_{\mbf{x}}^*}$. Therefore, when noise gets larger not only $\lambda_{\mbf{S}}^*$ is smaller than $\lambda_{\mbf{x}}^*$, but it also grows slower.
\end{remark}

\begin{remark}
Theorems \ref{theo:tik} and \ref{theo:interval} imply that, when regularization is accomplished through the Tikhonov method, $\lambda_{\mbf{x}}^*$ does not depend on the forward matrix $\mbf{G}$, while $\lambda_{\mbf{S}}^*$ does. The fact that $\lambda_{\mbf{x}}^*$ does not depend on $\mbf{G}$ may appear counter--intuitive: if the singular values grows, also the effective SNR of the data grow, and then the regularization parameter should become smaller. In fact, the regularization parameter does become smaller \textit{with respect to the data}; in other words, this is the classical behaviour of the optimal regularization parameter, where we are changing the SNR by increasing the strength of the exact signal, rather than decreasing the variance of the noise.
\end{remark}

\begin{remark}
When $M=N$ and $\sigma_1 = \dots = \sigma_M = 1$, the forward matrix $\mbf{G}$ is orthogonal and the inverse problem in equation (\ref{eq:model_eq}) is well--posed. Theorems \ref{theo:tik} and \ref{theo:interval} imply that  $\lambda_{\mbf{S}}^*$ and $\lambda_{\mbf{x}}^*$ are different also under these conditions, as 
\begin{displaymath}
h(\sigma_M) = h(\sigma_1) = -1 + \sqrt{1+\frac{\alpha^2}{\omega^2}} < \frac{\alpha^2}{\omega^2}
\end{displaymath}
Although unrealistic, this case is of particular interest in M/EEG functional connectivity because it corresponds to the ideal case where there is no \textit{cross--talk} or \textit{source--leakage} between sources \cite{ha_etal19}.

Indeed, in this case the resolution matrix is proportional to the identity matrix ($\mbf{R}_{\lambda} = \left(1+\lambda \right)^{-1} \mbf{I}$), i.e. the estimate at one location is not influenced by neural activity at different locations.
Our result shows that also in this ideal case the optimal values of the regularization parameters are different.
\end{remark}

\section{Beyond the two--step approach: Filter factor for a direct estimation of $\mbf{S}^{\mbf{x}}(f)$ from $\mbf{S}^{\mbf{y}}(f)$}\label{apx:single_step}

As an alternative to the two--step approach described so far, one may directly estimate the cross--power spectrum of the unknown $\mbf{S}^{\mbf{x}}(f)$ from that of the data $\mbf{S}^{\mbf{y}}(f)$. Indeed, from equation (\ref{eq:model_eq}) and from the linearity of the Fourier Transform it follows
\begin{equation}\label{eq:pb_cs}
\mathcal{S}^\mbf{y}(f)=(\mbf{G}\otimes\mbf{G})\mathcal{S}^\mbf{x}(f)+\mathcal{S}^\mbf{n}(f) ~~~,
\end{equation}
which describes a linear inverse problem.

Analogously to what we did in the previous Sections for the forward operator $\mbf{G}$, we can introduce the SVD of the forward operator $\mbf{G}\otimes\mbf{G}=(\mbf{U}\otimes\mbf{U})(\mbf{\Sigma}\otimes\mbf{\Sigma})(\mbf{V}\otimes\mbf{V})^t$, up to reordering the elements of $\mbf{\Sigma}\otimes\mbf{\Sigma}$ and the corresponding columns of $\mbf{U}\otimes\mbf{U}$ and $\mbf{V}\otimes\mbf{V}$. We can then express a \textit{one--step} regularized estimate of the cross--spectrum in terms of the SVD and of the filter factors
\begin{eqnarray}
\eqalign{
    \mathcal{S}^\mbf{x}_{\lambda}(f) & = \left(\mbf{V} \otimes \mbf{V} \right) \widetilde{\mbf{\Phi}}(\lambda)\left(\mbf{\Sigma} \otimes \mbf{\Sigma} \right){^{\dagger}} \left(\mbf{U} \otimes \mbf{U} \right)^t \mathcal{S}^\mbf{y}(f) \\
    & = \sum_{i,j}^{M} \widetilde{\varphi}_{i,j}(\lambda) \frac{\left(\mbf{u}_i \otimes \mbf{u}_j \right)^t \mathcal{S}^\mbf{y}(f)}{\sigma_i \sigma_j} \left(\mbf{v}_i \otimes \mbf{v}_j \right)~~~.
    }
\end{eqnarray}

In particular, if Tikhonov regularization is employed, the filter factors read
\begin{equation}
    \widetilde{\varphi}_{i,j}(\lambda) = \frac{\sigma_i^2 \sigma_j^2}{\sigma_i^2\sigma_j^2 + \lambda}
\end{equation}
while in tSVD the components such that the product $\sigma_i \sigma_j$ is below the threshold defined by $\lambda$ are filtered out.
Instead, in the classical two--step approach the filter factors for the estimated cross--spectrum are simply given by the product of the filter factors for the estimated source time--courses, that means each of the singular value $\sigma_i$ and $\sigma_j$ is individually filtered, instead of their product $\sigma_i \sigma_j$. Indeed, the cross--spectrum of the regularized estimate $\mbf{x}_{\lambda}(t)$ in equation  (\ref{eq:filter_factors}) is
\begin{eqnarray}
\eqalign{
    \mathcal{S}^{\mbf{x}_{\lambda}}(f) & = \left( \mbf{V} \otimes \mbf{V} \right) \left( \mbf{\Phi} \mbf{\Sigma}^{^{\dagger}} \otimes \mbf{\Phi} \mbf{\Sigma}^{^{\dagger}} \right) \left(\mbf{U} \otimes \mbf{U} \right)^t \mathcal{S}^\mbf{y}(f) \\
    & = \sum_{i,j}^{M} \varphi_i(\lambda) \varphi_j(\lambda) \frac{\left(\mbf{u}_i \otimes \mbf{u}_j \right)^t \mathcal{S}^\mbf{y}(f)}{\sigma_i \sigma_j} \left(\mbf{v}_i \otimes \mbf{v}_j \right) ~~~.
    }
\end{eqnarray}

As a comparison in Figures \ref{fig:filter_tSVD} and \ref{fig:filter_tikh} we plotted the filter factors $\widetilde{\varphi}_{i,j}(\lambda)$ and $\varphi_i(\lambda)\varphi_j(\lambda)$ for the tSVD and Tikhonov method. The forward matrix $\mbf{G}$ was a standard MEG forward operator based on a realistic, three--layer boundary element method (BEM) head model, publically available within the mne--python software \cite{gr_etal14}. For ease of representation only $M=20$ sensors and $N=25$ source locations were randomly selected. 

Figure \ref{fig:filter_tSVD} and \ref{fig:filter_tikh} highlight the potential advantages of the one--step approach over the two--step approach. In the case of tSVD, the filter factors of the two--step approach are zero whenever either $i < \lambda$ or  $j<\lambda$, which implies a jittering behaviour when plotted as a function of the product $\sigma_i \sigma_j$. In the one--step approach this issue is not present, because filtering is applied directly to the product of the singular values. In the case of the Tikhonov method we observe a similar behaviour, where in the one--step approach the filter factors increase smoothly when the product $\sigma_i \sigma_j$ increases, while in the two--step approach also higher values of such product may be severely filtered because of the effect of the regularization parameter on the individual singular values.

\begin{figure}
    \centering
    \includegraphics[scale=0.45]{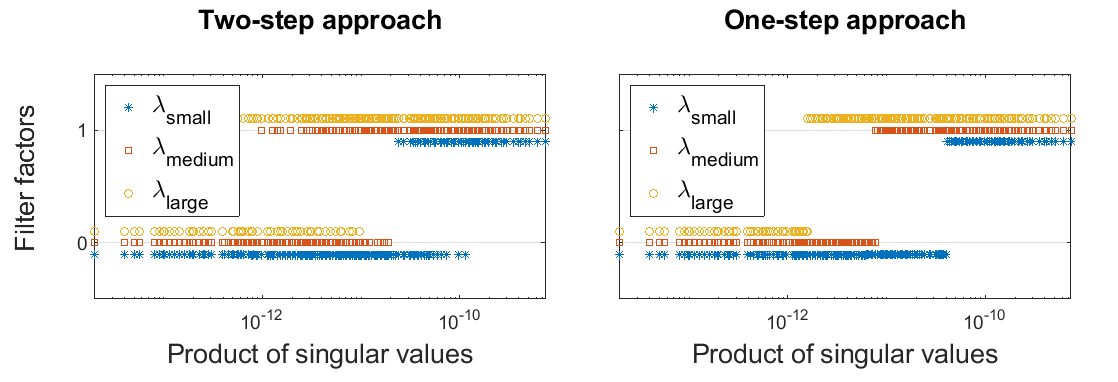}
    \caption{Filter factors for the tSVD method. On the x--axis the product of the singular values $\sigma_i\sigma_j$, on the y--axis the corresponding values of the filter factors $\varphi_i(\lambda)\varphi_j(\lambda)$ for the two--step approach (left) and $\widetilde{\varphi}_{i,j}(\lambda)$ for the one--step approach (right). The three different colors correspond to three different values of the regularization parameter, as illustrated in the legend. Please notice that the filter factors for tSVD are either zero or one, but different colors are plotted at slightly different levels for the sake of clarity.}
    \label{fig:filter_tSVD}
\end{figure}

\begin{figure}
    \centering
    \includegraphics[scale=0.45]{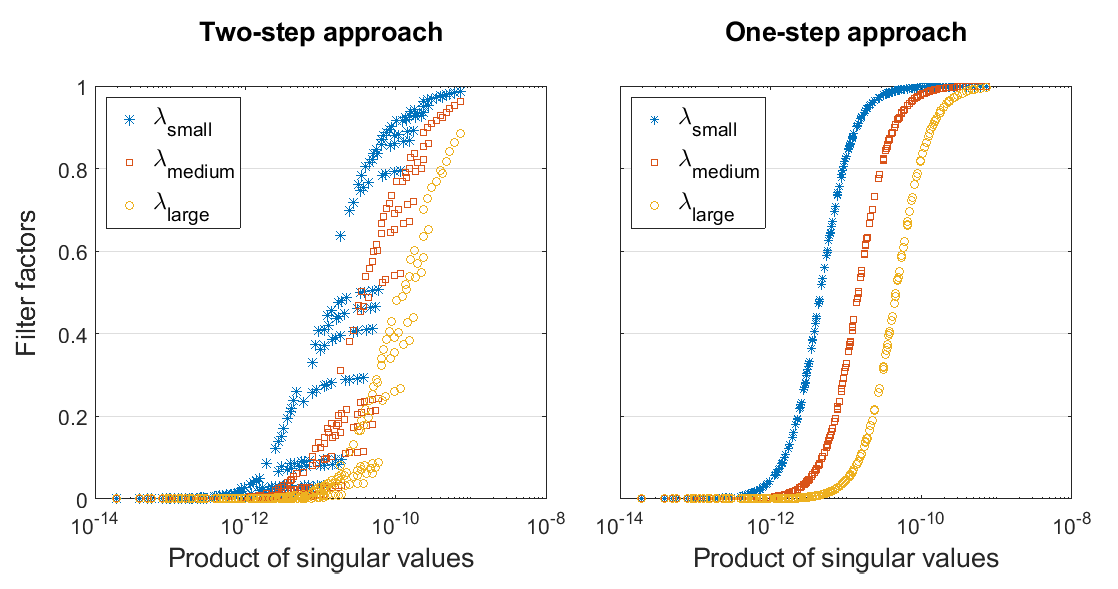}
    \caption{Filter factors for the Thikonov method. On the x--axis the product of the singular values $\sigma_i\sigma_j$, on the y--axis the corresponding values of the filter factors $\varphi_i(\lambda)\varphi_j(\lambda)$ for the two--step approach (left) and $\widetilde{\varphi}_{i,j}(\lambda)$ for the one--step approach (right). The three different colors correspond to three different values of the regularization parameter, as illustrated in the legend.}
    \label{fig:filter_tikh}
\end{figure}

\section{Discussion and future work}

Motivated by an analysis pipeline which is largely used for connectivity studies in the M/EEG community, in this article we have considered the problem of whether, in a two--step approach to the reconstruction of the cross--power spectrum of an unobservable signal, one should set the regularization parameter differently than what one would do for the reconstruction of the signal itself. 

First, making use of filter factor analysis, we obtained an explicit expression for the reconstruction error for the cross--power spectrum under the two--step approach. This formula is the analogous of the well--known formula for the reconstruction error in linear inverse problems, and holds in general. Then, under additional hypotheses of a white Gaussian signal and white Gaussian noise, we proved that the optimal values coincide for tSVD, while in the Tikhonov method the optimal value for the cross--spectrum is at most half the optimal value for the signal, thus proving also that the answer actually depends on the inverse method.

Our results are in line with the results of \cite{hietal16}, which showed empirically that the optimal estimate of connectivity is obtained with a regularization parameter smaller than the one providing the optimal estimate of the power spectrum, i.e. of the signal strength. Quantitatively, the recommendation in \cite{hietal16} was to use a parameter two orders of magnitude lower, while our main theorem for the Tikhonov method guarantees $\lambda_{\mbf{S}}^*<\lambda_{\mbf{x}}^*/2$. 

%However, there are several good reasons for this difference, including: (i) in their study, Hincapi\'e et al. used sinusoidal waveforms, while our proofs have been obtained for a signal with no temporal structure (white--noise--like); our result show that the optimal value for reconstruction of the cross--spectrum depends on the singular values of the forward matrix, that might have been relatively small in \cite{hietal16}; (iii) finally, and obviously, our result is only an upper bound.

Theorems \ref{theo:tsvd}, \ref{theo:tik} and \ref{theo:interval} have been obtained under the somewhat unrealistic assumption that the signal is a white--noise Gaussian process. While this is an important limitation with respect to the applications, preliminary numerical results indicate that the optimal value for the cross--spectrum is further reduced by the presence of a temporal structure in the signal; therefore, the inequality of our Theorem \ref{theo:tik} obtained in the ideal case would be strengthened in the application, in line with the mentioned results of \cite{hietal16}, which were obtained using sinusoidal signals.
In any case, future work will be devoted to investigating in detail the effect of a more plausible temporal structure of the input waveforms.

In addition, our results so far only concern the cross--power spectrum; future work will investigate the impact of the regularization parameter on the estimated value of the connectivity measure, such as Imaginary part of Coherency, Partial Directed Coherence and Granger causality.

Finally, as we point out in the last section, our results suggest that the two--step approach to estimation of the cross--power spectrum, and more in general of brain functional connectivity, might be sub--optimal.  This idea is in line with recent literature on the topic \cite{kietal08, cheu_etal10, fu_etal15, osetal18, su_etal17, tr_etal18}. Indeed, by looking at the filter factors obtained by the two--step approach, and comparing them to the filter factors one would get with a one--step approach to estimation of the power spectrum, we expect a better behaviour for this second option. Newly presented methods such as PSIICOS \cite{osetal18} present one--step approaches to the estimation of connectivity that benefit from this fact. Future work will be devoted to investigating more thoroughly this alternative approach.\\

\section*{Acknowledgments}
AS and MP have been partially supported by Gruppo Nazionale per il Calcolo Scientifico. SS kindly acknowledges Prof. Lauri Parkkonen and Dr. Narayan P. Subramaniyam for useful discussions. 
\section*{References}
\bibliographystyle{plain}
\bibliography{biblio}

\begin{thebibliography}{10}

\bibitem{baccala2001partial}
Luiz~A Baccal{\'a} and Koichi Sameshima.
\newblock Partial directed coherence: a new concept in neural structure
  determination.
\newblock {\em Biological cybernetics}, 84(6):463--474, 2001.

\bibitem{bekhti2018hierarchical}
Yousra Bekhti, Felix Lucka, Joseph Salmon, and Alexandre Gramfort.
\newblock A hierarchical bayesian perspective on majorization-minimization for
  non-convex sparse regression: application to m/eeg source imaging.
\newblock {\em Inverse Problems}, 34(8):085010, 2018.

\bibitem{be_pi11}
J~S Bendat and A~G Piersol.
\newblock {\em Random data: analysis and measurement procedures}, volume 729.
\newblock John Wiley \& Sons, 2011.

\bibitem{br_11}
M~J Brookes, M~Woolrich, H~Luckhoo, D~Price, J~R Hale, M~C Stephenson, G~R
  Barnes, S~M Smith, and P~G Morris.
\newblock Investigating the electrophysiological basis of resting state
  networks using magnetoencephalography.
\newblock {\em Proceedings of the National Academy of Sciences},
  108(40):16783--16788, 2011.

\bibitem{calvetti2015hierarchical}
Daniela Calvetti, Annalisa Pascarella, Francesca Pitolli, Erkki Somersalo, and
  Barbara Vantaggi.
\newblock A hierarchical krylov--bayes iterative inverse solver for meg with
  physiological preconditioning.
\newblock {\em Inverse Problems}, 31(12):125005, 2015.

\bibitem{ch_19}
F~Chella, L~Marzetti, M~Stenroos, L~Parkkonen, R~J Ilmoniemi, G~L Romani, and
  V~Pizzella.
\newblock The impact of improved {MEG}--{MRI} co-registration on {MEG}
  connectivity analysis.
\newblock {\em NeuroImage}, 197:354--367, 2019.

\bibitem{chella2014third}
Federico Chella, Laura Marzetti, Vittorio Pizzella, Filippo Zappasodi, and
  Guido Nolte.
\newblock Third order spectral analysis robust to mixing artifacts for mapping
  cross-frequency interactions in eeg/meg.
\newblock {\em Neuroimage}, 91:146--161, 2014.

\bibitem{cheu_etal10}
B~L~P Cheung, B~A Riedner, G~Tononi, and B~D Van~Veen.
\newblock Estimation of cortical connectivity from {EEG} using state-space
  models.
\newblock {\em IEEE Transactions on Biomedical engineering}, 57(9):2122--2134,
  2010.

\bibitem{costa2017bayesian}
Facundo Costa, Hadj Batatia, Thomas Oberlin, Carlos D'Giano, and Jean-Yves
  Tourneret.
\newblock Bayesian eeg source localization using a structured sparsity prior.
\newblock {\em NeuroImage}, 144:142--152, 2017.

\bibitem{de2010temporal}
F~De~Pasquale, S~Della~Penna, A~Z Snyder, C~Lewis, D~Mantini, L~Marzetti,
  P~Belardinelli, L~Ciancetta, V~Pizzella, and G~L et~al. Romani.
\newblock Temporal dynamics of spontaneous meg activity in brain networks.
\newblock {\em Proceedings of the National Academy of Sciences},
  107(13):6040--6045, 2010.

\bibitem{dePeMe96}
R~G de~Peralta~Menendez, S~L~G Andino, and B~L{\"u}tkenh{\"o}ner.
\newblock Figures of merit to compare distributed linear inverse solutions.
\newblock {\em Brain Topography}, 9(2):117--124, 1996.

\bibitem{fraschini2016effect}
M~Fraschini, M~Demuru, A~Crobe, F~Marrosu, C~J Stam, and A~Hillebrand.
\newblock The effect of epoch length on estimated eeg functional connectivity
  and brain network organisation.
\newblock {\em Journal of neural engineering}, 13(3):036015, 2016.

\bibitem{fr05}
Pascal Fries.
\newblock A mechanism for cognitive dynamics: neuronal communication through
  neuronal coherence.
\newblock {\em Trends in cognitive sciences}, 9(10):474--480, 2005.

\bibitem{fu_etal15}
M~Fukushima, O~Yamashita, T~R Kn{\"o}sche, and M~Sato.
\newblock {MEG} source reconstruction based on identification of directed
  source interactions on whole-brain anatomical networks.
\newblock {\em NeuroImage}, 105:408--427, 2015.

\bibitem{geweke1982measurement}
John Geweke.
\newblock Measurement of linear dependence and feedback between multiple time
  series.
\newblock {\em Journal of the American statistical association},
  77(378):304--313, 1982.

\bibitem{gr_etal14}
Alexandre Gramfort, Martin Luessi, Eric Larson, Denis~A Engemann, Daniel
  Strohmeier, Christian Brodbeck, Lauri Parkkonen, and Matti~S
  H{\"a}m{\"a}l{\"a}inen.
\newblock {MNE} software for processing {MEG} and {EEG} data.
\newblock {\em Neuroimage}, 86:446--460, 2014.

\bibitem{hail94}
M.~{H\"{a}m\"{a}l\"{a}inen} and R.~J. Ilmoniemi.
\newblock Interpreting magnetic fields of the brain: minimum norm estimates.
\newblock {\em Medical \& Biological Engineering \& Computing}, 32:35--42,
  1994.

\bibitem{hansen05}
P~C Hansen.
\newblock {\em Rank-deficient and discrete ill-posed problems: numerical
  aspects of linear inversion}, volume~4.
\newblock Siam, 2005.

\bibitem{hapu_17}
R~Hari and A~Puce.
\newblock {\em {MEG}-{EEG} Primer}.
\newblock Oxford University Press, 2017.

\bibitem{ha_etal19}
O~Hauk, M~Stenroos, and M~Treder.
\newblock Eeg/meg source estimation and spatial filtering: The linear toolkit.
\newblock {\em Magnetoencephalography: From Signals to Dynamic Cortical
  Networks}, pages 1--37, 2019.

\bibitem{hietal16}
AS~Hincapi{\'e}, J~Kujala, J~Mattout, S~Daligault, C~Delpuech, D~Mery,
  D~Cosmelli, and K~Jerbi.
\newblock Meg connectivity and power detections with minimum norm estimates
  require different regularization parameters.
\newblock {\em Computational Intelligence and Neuroscience}, 2016:19, 2016.

\bibitem{hietal17}
AS~Hincapi{\'e}, J~Kujala, J~Mattout, A~Pascarella, S~Daligault, C~Delpuech,
  D~Mery, D~Cosmelli, and K~Jerbi.
\newblock The impact of meg source reconstruction method on source-space
  connectivity estimation: a comparison between minimum-norm solution and
  beamforming.
\newblock {\em Neuroimage}, 156:29--42, 2017.

\bibitem{ilsa19}
R~J Ilmoniemi and J~Sarvas.
\newblock {\em {B}rain {S}ignals: {P}hysics and {M}athematics of {MEG} and
  {EEG}}.
\newblock Mit Press, 2019.

\bibitem{kietal08}
Stefan~J Kiebel, Marta~I Garrido, Rosalyn~J Moran, and Karl~J Friston.
\newblock Dynamic causal modelling for {EEG} and {MEG}.
\newblock {\em Cognitive neurodynamics}, 2(2):121, 2008.

\bibitem{li_etal15}
M~Liljestr{\"o}m, C~Stevenson, J~Kujala, and R~Salmelin.
\newblock Task-and stimulus-related cortical networks in language production:
  Exploring similarity of {MEG}-and {fMRI}-derived functional connectivity.
\newblock {\em Neuroimage}, 120:75--87, 2015.

\bibitem{luetal19}
Gianvittorio Luria, Dunja Duran, Elisa Visani, Sara Sommariva, Fabio Rotondi,
  Davide~Rossi Sebastiano, Ferruccio Panzica, Michele Piana, and Alberto
  Sorrentino.
\newblock Bayesian multi-dipole modelling in the frequency domain.
\newblock {\em Journal of neuroscience methods}, 312:27--36, 2019.

\bibitem{nolte04}
G~Nolte, O~Bai, L~Wheaton, Z~Mari, S~Vorbach, and M~Hallett.
\newblock Identifying true brain interaction from {EEG} data using the
  imaginary part of coherency.
\newblock {\em Clinical neurophysiology}, 115(10):2292--2307, 2004.

\bibitem{nolte2019mathematical}
G~Nolte, E~Galindo-Leon, Z~Li, X~Liu, and A~K Engel.
\newblock Mathematical relations between measures of brain connectivity
  estimated from electrophysiological recordings for gaussian distributed data.
\newblock {\em bioRxiv}, page 680678, 2019.

\bibitem{nolte08}
G~Nolte, A~Ziehe, V~V Nikulin, A~Schl{\"o}gl, N~Kr{\"a}mer, T~Brismar, and
  KR~M{\"u}ller.
\newblock Robustly estimating the flow direction of information in complex
  physical systems.
\newblock {\em Physical review letters}, 100(23):234101, 2008.

\bibitem{nuetal19}
Ronaldo~V Nunes, Marcelo~B Reyes, and Raphael~Y De~Camargo.
\newblock Evaluation of connectivity estimates using spiking neuronal network
  models.
\newblock {\em Biological cybernetics}, 113(3):309--320, 2019.

\bibitem{nunez99}
P~L Nunez, R~B Silberstein, Z~Shi, M~R Carpenter, R~Srinivasan, D~M Tucker, S~M
  Doran, P~J Cadusch, and R~S Wijesinghe.
\newblock {EEG} coherency {II}: experimental comparisons of multiple measures.
\newblock {\em Clinical Neurophysiology}, 110(3):469--486, 1999.

\bibitem{osetal18}
A~Ossadtchi, D~Altukhov, and K~Jerbi.
\newblock Phase shift invariant imaging of coherent sources {(PSIICOS)} from
  {MEG} data.
\newblock {\em NeuroImage}, 183:950--971, 2018.

\bibitem{pereda_05}
E~Pereda, R~Q Quiroga, and J~Bhattacharya.
\newblock Nonlinear multivariate analysis of neurophysiological signals.
\newblock {\em Progress in neurobiology}, 77(1):1--37, 2005.

\bibitem{sakkalis_11}
V~Sakkalis.
\newblock Review of advanced techniques for the estimation of brain
  connectivity measured with {EEG}/{MEG}.
\newblock {\em Computers in biology and medicine}, 41(12):1110--1117, 2011.

\bibitem{sc_gr19}
JM~Schoffelen and J~Gross.
\newblock Studying dynamic neural interactions with meg.
\newblock {\em Magnetoencephalography: from signals to dynamic cortical
  networks}, pages 1--23, 2019.

\bibitem{sommariva2019comparative}
S~Sommariva, A~Sorrentino, M~Piana, V~Pizzella, and L~Marzetti.
\newblock A comparative study of the robustness of frequency-domain
  connectivity measures to finite data length.
\newblock {\em Brain topography}, 32(4):675--695, 2019.

\bibitem{sorrentino2017inverse}
Alberto Sorrentino and Michele Piana.
\newblock Inverse modeling for meg/eeg data.
\newblock In {\em Mathematical and Theoretical Neuroscience}, pages 239--253.
  Springer, Cham, 2017.

\bibitem{st_10}
C~J Stam.
\newblock Use of magnetoencephalography ({MEG}) to study functional brain
  networks in neurodegenerative disorders.
\newblock {\em Journal of the Neurological Sciences}, 289(1-2):128--234, 2010.

\bibitem{su_etal17}
N~P Subramaniyam, F~Tronarp, S~S{\"a}rkk{\"a}, and L~Parkkonen.
\newblock Expectation--maximization algorithm with a nonlinear {K}alman
  smoother for {MEG/EEG} connectivity estimation.
\newblock In {\em EMBEC \& NBC 2017}, pages 763--766. Springer, 2017.

\bibitem{tr_etal18}
F~Tronarp, N~P Subramaniyam, S~S{\"a}rkk{\"a}, and L~Parkkonen.
\newblock Tracking of dynamic functional connectivity from {MEG} data with
  {K}alman filtering.
\newblock In {\em 2018 40th Annual International Conference of the IEEE
  Engineering in Medicine and Biology Society (EMBC)}, pages 1003--1006. IEEE,
  2018.

\bibitem{vanMi_etal19}
P~Van~Mierlo, Y~H{\"o}ller, N~K Focke, and S~Vulliemoz.
\newblock Network perspectives on epilepsy using {EEG/MEG} source connectivity.
\newblock {\em Frontiers in Neurology}, 10:721, 2019.

\bibitem{vvetal97}
B.~D. {Van~Veen}, W.~van Drongelen, M.~Yuchtman, and A.~Suzuki.
\newblock Localization of brain electrical activity via linearly constrained
  minimum variance spatial filtering.
\newblock {\em IEEE Transactions on Biomedical Engineering}, 44:867--880, 1997.

\bibitem{wahe_15}
D~G Wakeman and R~N Henson.
\newblock A multi-subject, multi-modal human neuroimaging dataset.
\newblock {\em Scientific data}, 2:150001, 2015.

\bibitem{welch1967use}
P~Welch.
\newblock The use of fast fourier transform for the estimation of power
  spectra: a method based on time averaging over short, modified periodograms.
\newblock {\em IEEE Transactions on audio and electroacoustics}, 15(2):70--73,
  1967.

\end{thebibliography}

\end{document}